\documentclass[12pt]{article}
\usepackage{amssymb}
\usepackage{amsmath,amsthm}
\usepackage[latin1]{inputenc}
\usepackage{graphicx}
\usepackage{color}
\newtheorem{definition}{\bf Definition}[section]
\newtheorem{Lem}[definition]{\bf Lemma}
\newtheorem{Thm}[definition]{\bf Theorem}

\newtheorem{Cor}[definition]{\bf Corollary}

\newtheorem{Rem}[definition]{\bf Remark}

\title{On global offensive $k$-alliances in graphs}

\author{Sergio Bermudo$^{1,}$\footnote{e-mail:\mbox{\tt
    sbernav\@@upo.es.} Partially supported by Ministerio de Ciencia y
Tecnolog\'ia, ref. BFM2003-00034 and Junta de Andaluc\'ia, ref.
FQM-260 and ref. P06-FQM-02225.},  Juan A.
Rodr\'{\i}guez-Vel\'{a}zquez$^{2,}$\footnote{e-mail:\mbox{\tt
juanalberto.rodriguez\@@urv.cat}. Partially supported by the Spanish
Ministry of Education through projects TSI2007-65406-C03-01
``E-AEGIS" and CONSOLIDER CSD2007-00004 ``ARES" and by the Rovira i
Virgili
 University through project 2006AIRE-09
} \\ Jos\'{e} M. Sigarreta$^{3,}$\footnote{e-mail:\mbox{\tt
    josemaria.sigarreta\@@uc3m.es}} and Ismael G.
Yero$^{2,}$\footnote{e-mail: ismael.gonzalez\@@urv.cat} \\
    \\
$^1${\small Department of Economy, Quantitative Methods and Economic
History} \\ {\small Pablo de Olavide University,  Carretera de
Utrera Km. 1, 41013-Sevilla, Spain}
\\
\\
$^2${\small Department of Computer Engineering and Mathematics
}\\
{\small  Rovira i Virgili University, Av. Pa\"{\i}sos Catalans 26,
43007 Tarragona, Spain.}
\\
\\
$^3$ {\small Faculty of Mathematics,  Autonomous University of
Guerrero}
\\
{\small Carlos E. Adame 5, Col. La Garita, Acapulco, Guerrero,
M\'{e}xico}}

\date{}

\begin{document}

\maketitle

\begin{abstract}

We investigate the relationship between global offensive
$k$-alliances and some characteristic sets of a graph including
$r$-dependent sets, $\tau$-dominating sets and standard dominating
sets. In addition, we discuss the close relationship that exist
among the (global) offensive
  $k_i$-alliance number of $\Gamma_i$, $i\in \{1,2\}$ and  the (global) offensive
   $k$-alliance number of $\Gamma_1 \times \Gamma_2$,
  for some specific values of $k$.
As a consequence of the study, we obtain bounds on the global
offensive $k$-alliance number in terms of several parameters of the
graph.
\end{abstract}



\section{Introduction}

The mathematical properties of alliances in graphs were first
studied by P. Kristiansen, S. M. Hedetniemi and S. T. Hedetniemi
\cite{KrHeHe}. They proposed different types of alliances that have
been extensively studied during the last four years. These types of
alliances are called \emph{defensive alliances}
\cite{KrHeHe,RoGoSi,RoSi6,SiRo1}, \emph{offensive alliances}
\cite{FaFrGoHeHeKrLaSk,RoSi3,SiRo2} and \emph{dual alliances} or
\emph{powerful alliances} \cite{BrDuHaHe,RoFeSi}. A generalization
of these alliances called $k$-alliances was presented by K. H.
Shafique and R. D. Dutton \cite{ShDu2,ShDu3}. We are interested in
the study of the mathematical properties of global offensive
$k$-alliances.


 We begin by stating the terminology
used. Throughout this article, $\Gamma = (V,E)$ denotes a simple
graph of order $|V | = n$. We denote two adjacent vertices $u$ and
$v$ by $u \sim v$. For a nonempty set $S \subseteq V$, and a vertex
$v \in V$, $N_S(v)$ denotes the set of neighbors $v$ has in $S$:
$N_S(v) := \{u \in S : u \sim v\}$, and the degree of $v$ in $S$
will be denoted by $\delta_S(v) = |N_S(v)|$. We denote the degree of
a vertex $v \in V$ by $\delta(v)$, the minimum degree of $\Gamma$ by
$\delta$ and the maximum degree by $\Delta$. The complement of the
set $S$ in $V$ is denoted by $\overline{S}$ and the boundary of $S$
is defined by $\partial(S) := \bigcup_{v\in S} N_{\overline{S}}(v)$.

A set $S \subseteq V$ is a \textit{dominating set} in $\Gamma$ if
for every vertex $v\in \overline{S}$,  $\delta_S(v)> 0$ (every
vertex in $\overline{S}$ is adjacent to at least one vertex in $S$).
The \textit{domination number} of $\Gamma$, denoted by
$\gamma(\Gamma)$, is the minimum cardinality of a dominating set in
$\Gamma$. For $k\in\{2-\Delta, . . . , \Delta\}$, a nonempty set $S
\subseteq V$ is an \textit{offensive $k$-alliance} in $\Gamma$ if
\begin{equation}\label{alliaceCondition}\delta_S(v)\geq
\delta_{\overline{S}}(v)+k, \quad \forall v \in \partial
(S)\end{equation} or, equivalently, \begin{equation}\delta(v)\geq
2\delta_{\overline{S}}(v)+k,\quad \forall v \in \partial (S)
.\end{equation} It is clear that if $k>\Delta$, no set $S$ satisfies
(\ref{alliaceCondition}) and, if $k<2-\Delta$, all the subsets of
$V$ satisfy it. An offensive $k$-alliance $S$ is called
\emph{global} if it is a dominating set. The \textit{offensive
$k$-alliance number} of $\Gamma$, denoted by $a_k^o(\Gamma)$, is
defined as the minimum cardinality of an offensive $k$-alliance in
$\Gamma$. The \textit{global offensive $k$-alliance number} of
$\Gamma$, denoted by $\gamma_k^o(\Gamma)$, is defined as the minimum
cardinality of a global offensive $k$-alliance in $\Gamma$. Notice
that $\gamma_k^o(\Gamma)\geq a_k^o(\Gamma)$ and
$\gamma_{k+1}^o(\Gamma)\geq \gamma_k^o(\Gamma)\ge \gamma(\Gamma)$.

In addition, if every vertex of $\Gamma$ has even degree and $k$ is
odd, $k = 2l-1$, then every global offensive $(2l-1)$-alliance in
$\Gamma$ is a global offensive $(2l)$-alliance. Hence, in such a
case, $\gamma_{2l-1}^o(\Gamma) = \gamma_{2l}^o(\Gamma)$.
Analogously, if every vertex of $\Gamma$ has odd degree and $k$ is
even, $k = 2l$, then every global offensive $(2l)$-alliance in
$\Gamma$ is a global offensive $(2l + 1)$-alliance. Hence, in such a
case, $\gamma_{2l}^o(\Gamma) = \gamma_{2l+1}^o(\Gamma)$.

\section{The global offensive $k$-alliance number for some families of graphs}

The problem of finding the global offensive $k$-alliance number is
NP-complete \cite{FeRoSi}. Even so, for some graphs  it is possible
to obtain this number. For instance, it is satisfied that for the
family of the complete graphs,  $K_n$, of order $n$
$$\gamma_k^o(K_n)=\left\lceil\frac{n+k-1}{2}\right\rceil ,$$
 for any cycle, $C_n$, of order $n$
$$
\gamma_k^o(C_n)=\left\{\begin{array}{c}
                         \vspace*{0.3cm}\left\lceil\frac{n}{3}\right\rceil, \quad \mbox{for $k=0,$}\,\,\,\,\\
                         \,\,\left\lceil\frac{n}{2}\right\rceil, \,\,\,\,\,\, \mbox{for $k=1,2$,}\\
                       \end{array}\right.
$$
and for any path, $P_n$, of order $n$
$$ \gamma_k^o(P_n)=\left\{\begin{array}{c}
                          \vspace*{0.3cm}\left\lceil\frac{n}{3}\right\rceil, \,\,\,\,\,\,\,\,\,\,\,\,\,\,\quad \mbox{for $k=0,$}\\
                         \left\lfloor\frac{n}{2}\right\rfloor+k-1, \quad \mbox{for $k=1,2.$}\\
                       \end{array}\right.
$$


\begin{Rem} Let $\Gamma=K_{r,t}$ be a complete bipartite graph with
$t\leq r$. For every $k\in\{2-r,...,r\}$,

\begin{itemize}

\item[\emph{(a)}] if $k\geq t+1$, then $\gamma_k^o(\Gamma)=r$.


\item[\emph{(b)}] if $k\leq t$ and $\left\lceil\frac{r+k}{2}\right\rceil+\left\lceil\frac{t+k}{2}\right\rceil\geq
t$, then $\gamma_k^o(\Gamma)=t$,


\item[\emph{(c)}] if $-t<k\le t$ and $\left\lceil\frac{r+k}{2}\right\rceil+\left\lceil\frac{t+k}{2}\right\rceil< t$,
then
$\gamma_k^o(\Gamma)=\left\lceil\frac{r+k}{2}\right\rceil+\left\lceil\frac{t+k}{2}\right\rceil$,

\item[\emph{(d)}] if $k\leq -t$ and $\left\lceil\frac{r+k}{2}\right\rceil+\left\lceil\frac{t+k}{2}\right\rceil< t$,
then
$\gamma_k^o(\Gamma)=\min\{t,1+\left\lceil\frac{r+k}{2}\right\rceil\}$.
\end{itemize}
\end{Rem}

\begin{proof}  (a) Let $\{V_t,V_r\}$ be the bi-partition of the vertex set of
$\Gamma$. Since $V_r$ is a global offensive $k$-alliance, we only
need to show that for every  global offensive $k$-alliance $S$,
$V_r\subseteq S$. If $v\in \overline{S}$ it satisfies
$\delta_S(v)\geq \delta_{\overline{S}}(v)+k>t$, in consequence $v\in
V_t$. Therefore, $\overline{S}\subseteq V_t$ or, equivalently,
$V_r\subseteq S$. Thus, we conclude that $\gamma_k^o(\Gamma)=r$.

(b)  If $k\le t$, it is clear that $V_t$ is a global offensive
$k$-alliance, then $\gamma_k^o(\Gamma)\le t$. We suppose that
$\left\lceil\frac{r+k}{2}\right\rceil+\left\lceil\frac{t+k}{2}\right\rceil\geq
t$ and there exists a global offensive $k$-alliance   $S=A\cup B$
such that $A\subseteq V_r$, $B\subseteq V_t$ and $|S|<t$.  In such a
case, as $S$ is a dominating set, $B\neq \emptyset$. Since $S$ is a
global offensive $k$-alliance, $ 2|B|\geq t+k $ and $ 2|A|\geq r+k.
$ Then we have,
$t>|S|\ge\left\lceil\frac{r+k}{2}\right\rceil+\left\lceil\frac{t+k}{2}\right\rceil\geq
t$, a contradiction.
 Therefore, $\gamma_k^o(\Gamma)=t$.

(c) In the proof of (b) we have shown that if there exists a global
offensive $k$-alliance $S$ of cardinality $|S|<t$, then $|S|\ge
\left\lceil\frac{r+k}{2}\right\rceil+\left\lceil\frac{t+k}{2}\right\rceil$.
Taking $A\subset V_r$ of cardinality
$\left\lceil\frac{r+k}{2}\right\rceil$ and $B\subset V_t$ of
cardinality $\left\lceil\frac{t+k}{2}\right\rceil$ we obtain a
global offensive $k$-alliance $S=A\cup B$ of cardinality
$|S|=\left\lceil\frac{r+k}{2}\right\rceil+\left\lceil\frac{t+k}{2}\right\rceil
$.

(d) Finally, if $S=A\cup B$ where $A\subseteq V_r$, $B\subseteq
V_t$, $|A|=\left\lceil\frac{r+k}{2}\right\rceil$ and $|B|=1$, then
$S$ is a global offensive $k$-alliance. Moreover, $S$ is a minimum
global offensive $k$-alliance if and only if
$|S|=1+\left\lceil\frac{r+k}{2}\right\rceil\leq t$.
\end{proof}

\section{Global offensive $k$-alliances and $r$-dependent sets}

A set $S \subseteq V$ is an \textit{$r$-dependent set} in $\Gamma$
if the maximum degree of a vertex in the subgraph $\langle S\rangle$
induced by $S$ is at most $r$, i.e., $\delta_S(v)\le r, \quad
\forall v \in S$. We denote by $\alpha_r(\Gamma)$ the maximum
cardinality of an $r$-dependent set in $\Gamma$ \cite{FaHeHe}.

\begin{Thm} Let $\Gamma$ be a graph of order $n$, minimum degree
$\delta$ and maximum degree $\Delta$.

\vspace{-0.2cm}

\begin{itemize}

\item[\emph{(a)}] If $S$ is an $r$-dependent set in $\Gamma$, $r\in\left\{0,...,\lfloor\frac{\delta-1}{2}\rfloor\right\}$, then $\overline{S}$ is a
 global offensive
$(\delta-2r)$-alliance.

\item[\emph{(b)}] If $S$ is a global offensive $k$-alliance in $\Gamma$, $k\in\{2-\Delta,...,\Delta\}$, then
$\overline{S}$ is a
$\left\lfloor\frac{\Delta-k}{2}\right\rfloor$-dependent set.

\vspace{-0.2cm}

\item[\emph{(c)}] Let  $\Gamma$ be a $\delta$-regular graph $(\delta>0)$.  $S$ is an $r$-dependent set in $\Gamma$, $r\in\left\{0,...,\lfloor\frac{\delta-1}{2}\rfloor\right\}$,
if and only if $\overline{S}$ is a global offensive
$(\delta-2r)$-alliance.

\end{itemize}

\end{Thm}

\begin{proof}(a) Let $S$ be an $r$-dependent set in $\Gamma$, then $\delta_S(v)\leq r$
for every $v\in S$. Therefore, $\delta_S(v)+\delta\leq
2\delta_S(v)+\delta_{\overline{S}}(v)\leq
2r+\delta_{\overline{S}}(v)$. As a consequence,
$\delta_{\overline{S}}(v)\geq \delta_S(v)+\delta-2r$, for every
$v\in S$. That is, $\overline{S}$ is a global offensive
$(\delta-2r)$-alliance in $\Gamma$.

(b) If $S$ is a global offensive $k$-alliance in $\Gamma$, then
$\delta(v)\geq 2\delta_{\overline{S}}(v)+k$ for every $v\in
\overline{S}$. As a consequence,
$\delta_{\overline{S}}(v)\leq\frac{\delta(v)-k}{2}\leq\frac{\Delta-k}{2}$
for every $v\in \overline{S}$, that is, $\overline{S}$ is a
$\left\lfloor\frac{\Delta-k}{2}\right\rfloor$-dependent set in
$\Gamma$.

(c) The result follows immediately from (a) and (b).
\end{proof}

\begin{Cor} Let $\Gamma$ be a graph of order $n$, minimum degree
$\delta $ and maximum degree $\Delta$.
\begin{itemize}
\item For every $k\in\{2-\Delta,...,\Delta \}$,  $
n-\alpha_{\left\lfloor\frac{\Delta-k}{2}\right\rfloor}(\Gamma)\leq\gamma_k^o(\Gamma).$

\item For every $k\in\{1,...,\delta\}$,
$ \gamma_k^o(\Gamma)\leq
n-\alpha_{\left\lfloor\frac{\delta-k}{2}\right\rfloor}(\Gamma). $

\item If $\Gamma$ is a $\delta$-regular graph $(\delta >0)$, for every $k\in\{1,...,\delta\}$,
$\gamma_k^o(\Gamma)= n-\alpha_{
\left\lfloor\frac{\delta-k}{2}\right\rfloor }(\Gamma).$
\end{itemize}
\end{Cor}

\section{Global offensive $k$-alliances and $\tau$-dominating sets}

Let  $\Gamma$ be a graph without   isolated vertices. For a given
$\tau\in (0,1]$, a set $S\subseteq V$ is called
\emph{$\tau$-dominating set} in $\Gamma$ if $\delta_S(v)\geq\tau
\delta(v)$ for every $v\in \overline{S}$. We denote by
$\gamma_{\tau}(\Gamma)$ the minimum cardinality of a
$\tau$-dominating set in $\Gamma$ \cite{DuHoLaMa}.

\begin{Thm} \label{alfas} Let $\Gamma$ be a graph of minimum degree $\delta>0$ and maximum degree
$\Delta$.

\begin{itemize}

\item[\emph{(a)}] If $0 < \tau \le \min\{  \frac{k+\delta}{2\delta},\frac{k+\Delta}{2\Delta}\}$,
then every global offensive $k$-alliance in $\Gamma$ is a
$\tau$-dominating set.

\vspace{0.2cm}

\item[\emph{(b)}] If  $ \max\{ \frac{k+\delta}{2\delta},\frac{k+\Delta}{2\Delta}\}\le
\tau \le 1$, then every $\tau$-dominating set in $\Gamma$ is a
global offensive $k$-alliance.

\end{itemize}

\end{Thm}

\begin{proof} (a) If $S$ is a global offensive $k$-alliance in $\Gamma$, then
$2\delta_S(v)\geq \delta(v)+k$ for every $v\in \overline{S}$.
Therefore, if $0 < \tau \le \min\{\frac{1}{2},
\frac{k+\delta}{2\delta}\}$, then $ \delta_S(v)\geq
\frac{\delta(v)+k}{2}\geq
 \frac{\delta(v)+\delta(2\tau-1)}{2}\geq\tau\delta(v).
$ Moreover, if $ \frac{1}{2}\le \tau \le \frac{k+\Delta}{2\Delta}$,
then $ \delta_S(v)\geq \frac{\delta(v)+k}{2}\geq
\frac{\delta(v)+\Delta(2\tau-1)}{2}\geq\tau\delta(v). $

(b) Since $\delta>0$, it is clear that every $\tau$-dominating set
is a dominating set. If $\tau \ge \frac{1}{2}$, then $
\delta(2\tau-1)\le \delta(v)(2\tau-1)$, for every vertex $v$ in
$\Gamma$. Hence, if $S$ is a $\tau$-dominating set and
$\frac{k+\delta}{2\delta}\le \tau$, we have $k\le (2\tau-1)\delta(v)
\le 2\delta_S(v)-\delta(v)$, for every $v\in \bar{S}$. Thus, $S$ is
a global offensive $k$-alliance in $\Gamma$.

On the other hand, if $\tau \le \frac{1}{2}$, then $
\Delta(2\tau-1)\le \delta(v)(2\tau-1)$, for every vertex $v$ in
$\Gamma$. Hence, if $S$ is a $\tau$-dominating set and
$\frac{k+\Delta}{2\Delta}\le \tau$, we have $k\le (2\tau-1)\delta(v)
\le 2\delta_S(v)-\delta(v)$, for every $v\in \bar{S}$. Thus, $S$ is
a global offensive $k$-alliance in $\Gamma$.
\end{proof}

\begin{Cor}
$S$ is a global offensive $(0)$-alliance  in $\Gamma$ if, and only
if, $S$ is a $(\frac{1}{2})$-dominating set.
\end{Cor}

\begin{Cor} $S$ is
a global offensive $k$-alliance in  a $\delta$-regular graph
$\Gamma$
 if, and only if, $S$ is a
$(\frac{k+\delta}{2\delta})$-dominating set in $\Gamma$.
\end{Cor}

\begin{Thm} \label{coro}Let $\Gamma$ be a graph of order $n$, minimum
degree $\delta>0$ and maximum degree $\Delta\ge 2$. For every
$j\in\{2-\Delta,...,0\}$ and $k\leq -\frac{j\delta}{\Delta}$ it is
satisfied $\gamma_k^o(\Gamma)+\gamma_j^o(\Gamma)\leq n.$
\end{Thm}

\begin{proof} If $j\in\{2-\Delta,...,0\}$, then there exists
$\tau\in\left[\frac{1}{\Delta},\frac{1}{2}\right]$ such that
$j=\Delta(2\tau-1)$. Therefore, if $S$ is a $\tau$-dominating set,
then (by Theorem \ref{alfas} (b)) $S$ is a global offensive
$j$-alliance. In consequence, $\gamma_j^o(\Gamma)\leq
\gamma_\tau(\Gamma)$. Moreover, if
$k\leq-\frac{j\delta}{\Delta}=\delta(1-2\tau)$, then $1-\tau\geq
\max\{\frac{1}{2}, \frac{k+\delta}{2\delta}\}$. Hence, by Theorem
\ref{alfas} (b), we have that every $(1-\tau)$-dominating set is a
global offensive $k$-alliance. Thus, $\gamma_k^o(\Gamma)\leq
\gamma_{1-\tau}(\Gamma)$. Using that
$\gamma_\tau(\Gamma)+\gamma_{1-\tau}(\Gamma)\leq n$ for $0<\tau<1$
(see Theorem 9 \cite{DuHoLaMa}), we obtain the required result.
\end{proof}

Notice that from Theorem \ref{coro} we have the following result.
\begin{Cor}\label{coro0-alliance} If $\Gamma$ is a graph of order $n$ and minimum degree $\delta>0$, then
$\gamma_0^o(\Gamma)\leq\frac{n}{2}$.
\end{Cor}

\section{Global offensive $k$-alliances and standard  dominating sets}

We say that a global offensive $k$-alliance $S$ is \emph{minimal} if
 no proper subset $S'\subset S$ is a global offensive $k$-alliance.

\begin{Thm} Let $\Gamma$ be a graph without isolated vertices and $k\leq
1$. If $S$ is a minimal global offensive $k$-alliance in $\Gamma$,
then $\overline{S}$ is a dominating set in $\Gamma$.
\end{Thm}

\begin{proof} We suppose there exists $u\in S$ such that
$\delta_{\overline{S}}(u)=0$ and let $S'=S\setminus\{u\}$. Since $S$
is a minimal global offensive $k$-alliance, and $\Gamma$ has no
isolated vertices, there exists $v\in \overline{S'}$ such that
$\delta_{S'}(v)<\delta_{\overline{S'}}(v)+k$. If $v=u$, we have
$\delta_{S}(u)=\delta_{S'}(u)<\delta_{\overline{S'}}(u)+k=k$, a
contradiction. If $v\neq u$, we have
$\delta_{S}(v)=\delta_{S'}(v)<\delta_{\overline{S'}}(v)+k=\delta_{\overline{S}}(v)+k$,
which is a contradiction too. Thus,
 $\delta_{\overline{S}}(u)>0$ for every $u\in S$.
 \end{proof}

 In the following result $\bar{\Gamma}=(V,\bar{E})$ denotes the
 complement of  $\Gamma=(V,E)$.

\begin{Lem} \label{OfenComplemento} Let $\Gamma$ be a graph of order $n$.
A dominating set $S$ in $\bar{\Gamma}$ is a global offensive
$k$-alliance in $\bar{\Gamma}$ if and only if $
\delta_S(v)-\delta_{\overline{S}}(v)+n+k-1 \leq 2|S| $  for every
$v\in \overline{S}$.
\end{Lem}

\begin{proof}  We know that a dominating set $S$ in $\bar{\Gamma}$ is a
global offensive $k$-alliance in $\bar{\Gamma}$ if and only if
$\bar{\delta}_S(v)\geq\bar{\delta}_{\overline{S}}(v)+k$ for every
$v\in \overline{S}$, where $\bar{\delta}_S(v)$ and
$\bar{\delta}_{\overline{S}}(v)$ denote the number of vertices that
$v$ has in $S$ and $\bar{S}$, respectively, in $\bar{\Gamma}$. Now,
using that $\bar{\delta}_S(v)=|S|-\delta_S(v)$ and
$\bar{\delta}_{\overline{S}}(v)=|\overline{S}|-1-\delta_{\overline{S}}(v)=n-|S|-1-\delta_{\overline{S}}(v)$,
we get that $S$ is a global offensive $k$-alliance in $\bar{\Gamma}$
if and only if $|S|-\delta_S(v)\geq
n-|S|-1+k-\delta_{\overline{S}}(v)$ or, equivalently, if
$\delta_S(v)-\delta_{\overline{S}}(v)+n+k-1\leq 2|S|$ for every
$v\in \overline{S}$. \end{proof}

\begin{Thm} Let $\Gamma$ be a graph of order $n$, minimum degree $\delta$ and maximum degree
$\Delta$.

\begin{itemize}
\item[\emph{(a)}] Every  dominating set in
$\bar{\Gamma}=(V,\bar{E})$, $S\subseteq V$, of cardinality
$|S|\geq\left\lceil\frac{n+k+\Delta-1}{2}\right\rceil$ is a global
offensive $k$-alliance in $\bar{\Gamma}$.

\item[\emph{(b)}] Every  dominating set in $\Gamma=(V,E)$,
$S\subseteq V$, of cardinality
$|S|\geq\left\lceil\frac{2n+k-\delta-2}{2}\right\rceil$ is a global
offensive $k$-alliance in $\Gamma$.
\end{itemize}
\end{Thm}

\begin{proof}  If $S$ is a dominating set in $\bar{\Gamma}$ and it
satisfies $|S|\geq\left\lceil\frac{n+k+\Delta-1}{2}\right\rceil$,
then
$$
|S|\geq\frac{n+k+\Delta-1}{2}\geq\frac{\delta_S(v)-\delta_{\overline{S}}(v)+n+k-1}{2}
$$
for every vertex $v$. Therefore, by Lemma \ref{OfenComplemento} we
have that $S$ is a global offensive $k$-alliance in $\bar{\Gamma}$.
Thus, the result (a) follows.

Analogously, by replacing $\Gamma$ by $\bar{\Gamma}$ and taking into
account that the maximum degree in $\bar{\Gamma}$ is $n-1-\delta$,
the result (b) follows.
\end{proof}

\section{The Cartesian product of $k$-alliances}

In  this section we discuss the close relationship that exist among
the (global) offensive
  $k_i$-alliance number of $\Gamma_i$, $i\in \{1,2\}$ and  the (global) offensive $k$-alliance number of $\Gamma_1 \times \Gamma_2$,
  for some specific values of $k$.

\begin{Thm}\label{teo-k-Delta-1o2}
Let  $\Gamma_i=(V_i,E_i)$ be a graph of minimum degree
$\delta_i$ and maximum degree $\Delta_i$, $i\in \{1,2\}$.

\begin{itemize}

\item[$(a)$] If
$S_i$ is an offensive $k_i$-alliance in $\Gamma_i$, $i\in \{1,2\}$,
then, for $k=\min\{k_2-\Delta_1,k_1-\Delta_2\}$,  $S_1\times S_2$ is
an offensive $k$-alliance in $\Gamma_1\times\Gamma_2$.

\item[$(b)$] Let $S_i\subset V_i$, $i\in \{1,2\}$. If  $S_1\times S_2$ is an offensive $k$-alliance in
$\Gamma_1\times\Gamma_2$,  then $S_1$ is an offensive
$(k+\delta_2)$-alliance in $\Gamma_1$  and $S_2$ is an offensive
$(k+\delta_1)$-alliance in $\Gamma_2$, moreover, $k\le
\min\{\Delta_1-\delta_2,\Delta_2-\delta_1\}$.
\end{itemize}
\end{Thm}

\begin{proof}
If $X=S_1\times S_2$, then $(u,v)\in
\partial X$ if and only if, either $u\in \partial S_1$ and
$v\in S_2$ or $u\in S_1$ and $v\in \partial S_2$. We differentiate
two cases:

\begin{itemize}
\item[]Case 1: If $u\in \partial S_1$ and
$v\in S_2$, then $\delta_X(u,v)=\delta_{S_1}(u)$ and
$\delta_{\overline{X}}(u,v)=\delta_{\overline{S_1}}(u)+\delta(v)$.
\item[]Case 2: If $u\in S_1$ and $v\in \partial S_2$, then  $\delta_X(u,v)=\delta_{S_2}(v)$ and
$\delta_{\overline{X}}(u,v)=\delta(u)+\delta_{\overline{S_2}}(v)$.
\end{itemize}

\begin{itemize}
\item[$(a)$]{ In Case 1 we have $\delta_X(u,v)=\delta_{S_1}(u)\ge
\delta_{\overline{S_1}}(u)+k_1=\delta_{\overline{X}}(u,v)-\delta(v)+k_1\ge
\delta_{\overline{X}}(u,v)-\Delta_2+k_1$ and in  Case 2 we obtain
$\delta_X(u,v)=\delta_{S_2}(v)\ge
\delta_{\overline{S_2}}(v)+k_2=\delta_{\overline{X}}(u,v)-\delta(u)+k_2\ge
\delta_{\overline{X}}(u,v)-\Delta_1+k_2.$ Hence, for every $(u,v)\in
\partial X$, $\delta_X(u,v)\ge \delta_{\overline{X}}(u,v)+k$, with
$k=\min\{k_2-\Delta_1,k_1-\Delta_2\}$. So, the result follows.}
\item[$(b)$]{
In  Case 1 we have $\delta_{S_1}(u)=\delta_X(u,v)\ge
\delta_{\overline{X}}(u,v)+k=\delta_{\overline{S_1}}(u)+\delta(v)+k=
\delta_{\overline{S_1}}(u)+\delta_2+k$ and in  Case 2 we deduce
$\delta_{S_2}(v)=\delta_X(u,v)\ge
\delta_{\overline{X}}(u,v)+k=\delta_{\overline{S_2}}(v)+\delta(u)+k\ge
\delta_{\overline{S_2}}(v)+\delta_1+k.$ Hence, for every $u\in
\partial S_1$, $\delta_{S_1}(u)\ge
\delta_{\overline{S_1}}(u)+\delta_1+k$ and for every $v\in \partial
S_2$, $\delta_{S_2}(v)\ge \delta_{\overline{S_2}}(v)+\delta_2+k$.
So, the result follows.}
\end{itemize}

\end{proof}

\begin{Cor}
Let  $\Gamma_i$ be  a graph of maximum degree $\Delta_i$,
$i\in\{1,2\}$. Then  for  every $k\le
\min\{k_1-\Delta_2,k_2-\Delta_1\}$,
 $a_{k}^o(\Gamma_1\times\Gamma_2)\le
a_{k_1}^o(\Gamma_1)a_{k_2}^o(\Gamma_2)$.
\end{Cor}

For the particular case of the graph $C_4\times K_4$, we have
$a_{-3}^o(C_4\times K_4)=2=a_{0}^o(C_4)a_{1}^o(K_4)$.



\begin{Thm}\label{prodCartGlobal} Let $\Gamma_2=(V_2,E_2)$ be a
graph of maximum degree $\Delta_2$ and minimum degree $\delta_2$.

\begin{itemize}
\item[$(i)$]{If $S$ is a global offensive $k$-alliance in $\Gamma_1$, then
$S\times V_2$ is a global offensive $(k-\Delta_2)$-alliance in
$\Gamma_1\times \Gamma_2$.}
\item[$(ii)$]{If $S\times V_2$ is a global offensive $k$-alliance in
$\Gamma_1\times \Gamma_2$, then $S$ is a global offensive
$(k+\delta_2)$-alliance in $\Gamma_1$, moreover, $k\le
\Delta_1-\delta_2$, where $\Delta_1$ denotes the maximum degree of
$\Gamma_1$.}
\end{itemize}
\end{Thm}

\begin{proof}$\;$\begin{itemize}
\item[$(i)$]{
We first note that, as $S$ is a dominating set in $\Gamma_1$,
$X=S\times V_2$ is a dominating set in $\Gamma_1\times \Gamma_2$. In
addition,  for every $x_{ij}=(u_i,v_j)\in \bar{X}$ we have
$\delta_{X}(x_{ij})=\delta_{S}(u_{i})$ and
$\delta_{\bar{S}}(u_{i})+\Delta_2\ge
\delta_{\bar{S}}(u_{i})+\delta(v_j)=\delta_{\bar{X}}(x_{ij})$, so
$\delta_{X}(x_{ij})=\delta_{S}(u_{i})\ge
\delta_{\bar{S}}(u_{i})+k\ge \delta_{\bar{X}}(x_{ij})-\Delta_2+k$.
Thus, $X$ is a global offensive $(k-\Delta_2)$-alliance in
$\Gamma_1\times \Gamma_2$.}
\item[$(ii)$]{From Theorem \ref{teo-k-Delta-1o2} (a) we obtain that $S$ is an offensive
$(k+\delta_2)$-alliance in $\Gamma_1$ and $k\le \Delta_1-\delta_2$.
We only need to show that $S$ is a dominating set. As $S\times V_2$
is a dominating set in $\Gamma_1\times \Gamma_2$, we have that for
every $u\in \overline{S}$ and $v\in V_2$ there exists $(a,b)\in
S\times V_2$ such that  $(a,b)$ is adjacent to $(u,v)$, hence, $b=v$
and $a$ is adjacent to $u$, so the result follows.}
\end{itemize}
\end{proof}
It is easy to see the following result on domination,
$\gamma(\Gamma_1\times \Gamma_2 )\le n_2 \gamma(\Gamma_1)$, where
$n_2$ is the order of $\Gamma_2$. An ``analogous" result on global
offensive $k$-alliances can be deduced from Theorem
\ref{prodCartGlobal} $(i)$.

\begin{Cor} \label{CoroCartesianGlobal} For any graph $\Gamma_1$ and any graph $\Gamma_2$ of order $n_2$ and maximum degree $\Delta_2$,
 $\gamma_{k-\Delta_2}^o(\Gamma_1\times \Gamma_2 )\le n_2 \gamma_{k}^o(\Gamma_1).$
\end{Cor}

We emphasize the following particular cases of Corollary
\ref{CoroCartesianGlobal}.

\begin{Rem}\label{prodCart} For any graph $\Gamma$,
\begin{itemize}
\item[{\rm (a)}] $\gamma_{k-2}^o(\Gamma\times C_t)\le t \gamma_{k}^o(\Gamma),$
\item[{\rm (b)}] $\gamma_{k-2}^o(\Gamma\times P_t)\le t \gamma_{k}^o(\Gamma).$
\item[{\rm (c)}] $\gamma_{k-t+1}^o(\Gamma\times K_t)\le t \gamma_{k}^o(\Gamma).$
\end{itemize}
\end{Rem}

Notice also that if $\Gamma_2$ is a regular graph, Theorem
\ref{prodCartGlobal} $(i)$ can be simplified as follow.

\begin{Cor}
Let $\Gamma_2=(V_2,E_2)$ be a $\delta$-regular graph. A set $S$ is a
global offensive $k$-alliance in $\Gamma_1$ if and only if $S\times
V_2$ is a global offensive $(k-\delta)$-alliance in $\Gamma_1\times
\Gamma_2$.
\end{Cor}

\section{Bounding the global offensive $k$-alliance number}

In general, the problem of finding the global offensive $k$-alliance
number is NP-complete \cite{FeRoSi}. In the following results we
obtain some bounds on this number involving some other parameters of
the graphs.

\begin{Rem} For every $k\in \{4-n,...,n-1\}$,
\begin{itemize}
\item[{\rm (a)}] $\left\lceil\frac{t(n+k-3)}{2}\right\rceil\le \gamma_k^o(K_n\times
C_t) \le t\left\lceil\frac{n+k+1}{2}\right\rceil$,
\item[{\rm (b)}] $\left\lceil\frac{t(n+k-3)+2}{2}\right\rceil\le \gamma_k^o(K_n\times
P_t) \le t\left\lceil\frac{n+k+1}{2}\right\rceil$
\end{itemize}
\end{Rem}

\begin{proof}$\;$
(a) Let $S=\cup_{i=1}^tS_i\subset V(K_n\times C_t)$, where each
$S_i$ ($1\le i\le t$) is a subset of each one of the $t$ copies of
$K_n$, respectively. If $S$ is a global offensive $k$-alliance in
$K_n\times C_t$, then for every $v\in \bar{S}$ we have, $|S_i|+2\ge
\delta_{S}(v)\ge \delta_{\bar{S}}(v)+k\ge n-1-|S_i|+k,$ where $S_i$
is the corresponding subset of $S$ included in the same copy of
$K_n$ containing $v$. Thus, $|S_i|\ge \frac{n+k-3}{2}$. Hence, for
$k>3-n$, we obtain that $S_i \neq \emptyset$. Therefore,
$|S|=\sum_{i=1}^t|S_i|\ge \frac{t(n+k-3)}{2}$. The upper bound is
obtained directly from Remark \ref{prodCart}. The proof of (b) is
completely analogous.
\end{proof}

\begin{Thm} \label{CotasOfen} Let $\Gamma$ be a graph of order $n$, size $m$,
minimum degree $\delta$ and maximum degree $\Delta$. For every
$k\in\{2-\delta,...,\delta\}$, the following inequality holds
$$
\gamma_k^o(\Gamma)\ge \left\lceil\frac{(n+2\Delta+k)-
\sqrt{(n+2\Delta+k)^2-4(2m+kn)}}{2}\right\rceil.
$$
\end{Thm}

\begin{proof}
To get the  bound, we know that
\begin{align*}
2m &= \sum_{v\in S}\delta(v)+\sum_{v\in
\overline{S}}\delta_S(v)+\sum_{v\in
\overline{S}}\delta_{\overline{S}}(v) \\ &\leq 2|S|\Delta+\sum_{v\in
\overline{S}}(\delta_S(v)-k)\\ &\leq 2|S|\Delta+(n-|S|)(|S|-k).
\end{align*}
Then, the result follows by solving the inequality
$|S|^2-(n+2\Delta+k)|S|+2m+kn\leq 0$.
\end{proof}

Notice the  bound is tight, if we consider the cube $Q_3$   we
obtain $\gamma_{-1}^o(Q_3)=\gamma(Q_3)=2$ and
$\gamma_{2}^o(Q_3)=\gamma_{3}^o(Q_3)=4$.

The upper bound in the following theorem have been correctly
obtained in \cite{FeRoSi} but it appears in the article with a
mistake, should be $\left\lfloor\frac{\delta-k+2}{2}\right\rfloor$
instead of $\left\lceil\frac{\delta-k+2}{2}\right\rceil.$ So in this
article we include the correct result without proof. The lower bound
is an immediate generalization of the previous results obtained in
\cite{SiRo2} for $k=1$ and $k=2$.

\begin{Thm} \label{perdodo} Let $\Gamma$ be a graph of order $n$, size $m$ and maximum degree $\Delta$.
Then $$ \left\lceil\frac{2m+kn}{3\Delta+k}\right\rceil   \le
\gamma_{k}^{o}(\Gamma)  \le
n-\left\lfloor\frac{\delta-k+2}{2}\right\rfloor$$
\end{Thm}



The   bounds of Theorem \ref{perdodo} are tight. For instance, the
 lower bound is attained in the case of the 3-cube, $Q_3$, for
every $k$.
The upper bound is attained, for instance, for the complete graph,
$K_{n}$, for every $k$, i.e., $\gamma_{k}^{o}(K_{n}) =
\left\lceil\frac{n+k-1}{2}\right\rceil.$

There are graphs in which Theorem \ref{CotasOfen} leads to better
results than the lower bound  in Theorem \ref{perdodo} and
viceversa. For instance, for $k=1$  and $\Gamma=K_5$ the bound in
Theorem \ref{CotasOfen} is attained but the lower bound in Theorem
\ref{perdodo} is not. The opposite one  occurs for the case of the
3-cube graph.


\begin{Cor}
Let ${\cal L}(\Gamma)$ be the line graph of a simple graph $\Gamma$
of size $m$. Let $\delta_{1}\geq \delta_{2}\geq \cdots
\geq\delta_{n}$ be the degree sequence of $\Gamma$. Then
$$\gamma_{k}^{o}({\cal L}(\Gamma))\ge
\left\lceil\frac{\displaystyle{\sum_{i=1}^{n}\delta_{i}^{2}}+m(k-2)}{3(\delta_{1}+\delta_{2}-2)+k}\right\rceil.$$
\end{Cor}



\begin{Thm} Let $\Gamma$ be a graph of order $n$ and maximum degree $\Delta$. For all global offensive
$k$-alliance $S$ in $\Gamma$ such that the subgraph $\langle
\overline{S}\rangle$ has minimum degree $p$,
$|S|\geq\left\lceil\frac{(p+k)n}{\Delta+p+k}\right\rceil$.
\end{Thm}

\begin{proof} The number of edges in the subgraph $\langle
\overline{S}\rangle$ satisfies $m(\langle
\overline{S}\rangle)\geq\frac{(n-|S|)p}{2}$, hence,
$$
\Delta|S|\!\geq\!\sum_{v\in
\overline{S}}\delta_S(v)\!\geq\!\sum_{v\in
\overline{S}}\delta_{\overline{S}}(v)+k(n-|S|)\!=\!2m(\langle
\overline{S}\rangle)+k(n-|S|)\!\geq\!(p+k)(n-|S|),
$$
in consequence, $|S|\geq\frac{(p+k)n}{\Delta+p+k}$.
\end{proof}

\vspace{0.2cm}

 Notice the bound is attained for the minimal
global offensive $k$-alliance in the case of the 3-cube graph $Q_3$
for $k = -1, 2, 3.$ For $k =-1$ we have $|S| = 2$ and $p = 2$, and
for $k = 2, 3$ we have $|S| = 4$ and $p = 0$.

\end{document}